\documentclass[12pt, reqno]{amsart}

\usepackage{amsmath,amssymb,txfonts}
\usepackage{amssymb}
\usepackage{amsxtra}
\usepackage{amsthm, color}
\usepackage{txfonts}
\usepackage{graphicx}
\usepackage{color}
\usepackage{times}

\usepackage[cp1252]{inputenc}
\usepackage{mathrsfs}

\newcommand{\rmv}[1]{}

\newcommand{\HH}{\mathcal{H}}

\newcommand{\HM}{\mathcal{M}}

\newcommand{\HF}{\mathcal{F}}

\newcommand{\HL}{\mathcal{L}}
\newcommand{\HD}{\mathcal{D}}

\newcommand{\HB}{\mathcal{B}}

\newcommand{\D}{\mathbb{D}}

\newcommand{\C}{\mathbb{C}}

\newcommand{\T}{\mathbb{T}}

\newcommand{\intD}{\int_{\D}}

\allowdisplaybreaks

\theoremstyle{plain}
\newtheorem{theorem}{Theorem}
\newtheorem{lemma}[theorem]{Lemma}

\newtheorem{prop}[theorem]{Proposition}

\newtheorem{definition}[theorem]{Definition}

\theoremstyle{definition}

\begin{document}

\title{Sub-Bergman Hilbert spaces on the unit disk III}

\author{Shuaibing Luo}
\address{Shuaibing Luo, School of Mathematics, Hunan University, Changsha, Hunan 410082, China}
\email{sluo@hnu.edu.cn}

\author{Kehe Zhu}
\address{Kehe Zhu, Department of Mathematics and Statistics, State University of New York, Albany, NY 12222, USA}
\email{kzhu@albany.edu}


\thanks{Luo was supported by National Natural Science Foundation of China (No. 12271149). Zhu was supported by National Natural Science Foundation of China (No. 12271328).}
\date{}

\begin{abstract}
For a bounded analytic function $\varphi$ on the unit disk $\D$ with $\|\varphi\|_\infty\le1$ we consider the
defect operators $D_\varphi$ and $D_{\overline\varphi}$ of the Toeplitz operators $T_\varphi$ and
$T_{\overline\varphi}$, respectively, on the weighted Bergman space $A^2_\alpha$. The ranges of $D_\varphi$
and $D_{\overline\varphi}$, written as $H(\varphi)$ and $H(\overline\varphi)$ and equipped with appropriate
inner products, are called sub-Bergman spaces.

We prove the following three results in the paper: for $-1<\alpha\le0$ the space $H(\varphi)$ has a complete
Nevanlinna-Pick kernel if and only if $\varphi$ is a M\"{o}bius map; for $\alpha>-1$ we have
$H(\varphi)=H(\overline\varphi)=A^2_{\alpha-1}$ if and only if the defect operators $D_\varphi$ and
$D_{\overline\varphi}$ are compact; and for $\alpha>-1$ we have $D^2_\varphi(A^2_\alpha)=
D^2_{\overline\varphi}(A^2_\alpha)=A^2_{\alpha-2}$ if and only if $\varphi$ is a finite Blaschke product.
In some sense our restrictions on $\alpha$ here are best possible.
\end{abstract}

\maketitle


\section{Introduction}

Let $\HH$ be a Hilbert space and $B(\HH)$ be the space of all bounded linear operators on $\HH$. If $T\in B(\HH)$ is
a contraction, we use $H(T)$ to denote the range space of the defect operator $(I-TT^*)^{1/2}$. It is well known that
$H(T)$ is a Hilbert space with the inner product
$$\langle (I-TT^*)^{1/2} x, (I-TT^*)^{1/2}y\rangle_{H(T)} = \langle x, y\rangle_\HH,$$
where $x, y \in \HH \ominus \ker(I-TT^*)^{1/2}$. Spaces of the type $H(T)$ have been studied extensively in the
literature, mostly in connection with operator models. 

There are two special cases that are especially interesting. First, if $\HH=H^2$ is the classical Hardy space on the
unit disk $\D$, and if $T=T_\varphi$ is the analytic Toeplitz operator (multiplication operator) induced a function
$\varphi$ in the unit ball $H^\infty_1$ of $H^\infty$, then $H(T_\varphi)$ is called a sub-Hardy space (or
a de Branges-Rovnyak space). Such spaces appeared in the work \cite{deB85} of de Branges concerning the
Bieberbach conjecture and were studied systematically in Sarason's monograph \cite{Sa94}.
See also the recent monograph \cite{FM16}.

Second, if $\HH=A^2$ is the classical Bergman space on the unit disk and if $T=T_\varphi$ is the analytic Toeplitz
operator (multiplication operator) on $A^2$ for some $\varphi\in H^\infty_1$, then $H(T_\varphi)$ is naturally called
a sub-Bergman space. Such spaces have been studied by several authors in the literature, beginning
with \cite{Zhu96, Zhu03} and including \cite{AJ10, Chu18, Chu19, Chu20, GHLP, Su06, Sy10}.

In this paper we focus on sub-Bergman spaces in the weighted case. More specifically, we will consider a family
of ``generalized Bergman spaces'' $A^2_\alpha$.
We will also consider multiplications operators $T_\varphi=T_\varphi^\alpha: A^2_\alpha\to A^2_\alpha$ induced
by functions from $\HM_1(A^2_\alpha)$, the closed unit ball of the multiplier algebra $\HM(A^2_\alpha)$ of
$A^2_\alpha$. It is natural for us to use the notation $H^\alpha(\varphi)$ for the space $H(T_\varphi)$. Similarly,
we will write $H^\alpha(\overline\varphi)$ for the space $H(T)$ when $T$ is the adjoint operator
$T^*_\varphi: A^2_\alpha\to A^2_\alpha$. Note that for $\alpha\ge-1$ we have $\HM(A^2_\alpha)=H^\infty$.

Motivated by the main results obtained in \cite{Zhu03, Chu20}, we will study the following three problems:
\begin{itemize}
\item[(a)] When does $H^\alpha(\varphi)$ have a complete Nevanlinna-Pick (CNP)
kernel? For $-1<\alpha\le0$ we show that $H^\alpha(\varphi)$ has a CNP kernel if and only if $\varphi$ is a M\"{o}bius
map. A (more subtle) characterization is also obtained when $-2<\alpha<-1$. The upper bound $0$ and lower
bound $-2$ for $\alpha$ appear to be natural obstacles. Here, even the result for the case $\alpha=0$ is new.
The case $\alpha=-1$ was studied in \cite{Chu20}.
\item[(b)] When do we have
$$H^\alpha(\varphi)=H^\alpha(\overline\varphi)=A^2_{\alpha-1}?$$
For $\alpha>-1$ we show that the identities above hold if and only if $\varphi$ is a finite Blaschke product, which is
also equivalent to the corresponding defect operators being compact. Our methods rely on the assumption $\alpha>-1$
in a very critical way. In particular, our main results definitely become invalid when $\alpha=-1$ (the Hardy space case).
Some special cases of our main results for this problem can be found in \cite{Zhu03, Su06, AJ10, Chu18, Chu19, GHLP}.
\item[(c)] When do we have
$$(I-T_\varphi T_{\overline\varphi})(A^2_\alpha)=(I-T_{\overline\varphi}T_\varphi)(A^2_\alpha)=A^2_{\alpha-2}?$$
We show that, for $\alpha>-1$, the identities above hold if and only if $\varphi$ is a finite Blaschke product. The
special case $\alpha=0$ was proved in \cite{Zhu03}. Once again, the assumption $\alpha>-1$ is critical here.
\end{itemize}

With the definition of generalized Bergman spaces $A^2_\alpha$ deferred to the next section, we mention the
following special cases: $A^2_0=A^2$ is the ordinary Bergman space, $A^2_{-1}=H^2$ is the Hardy space,
and $A^2_{-2}=\HD$ is the Dirichlet space.

\section{Generalized Bergman spaces}

For any real number $\alpha$ we fix some non-negative integer $k$ such that $2k+\alpha>-1$ and let $A^2_\alpha$
denote the space of analytic functions $f$ on $\D$ such that
\begin{equation}
\intD(1-|z|^2)^{2k}|f^{(k)}(z)|^2\,dA_\alpha(z)<\infty,
\label{eq1}
\end{equation}
where
$$dA_\alpha(z)=(1-|z|^2)^\alpha\,dA(z).$$
Here $dA$ is the normalized area measure on $\D$. It is easy to see that the weighted area measure $dA_\alpha$ is
finite if and only if $\alpha>-1$, in which case we will normalize $dA_\alpha$ so that $A_\alpha(\D)=1$.

It is well known that the space $A^2_\alpha$ is independent of the choice of the integer $k$ used in
(\ref{eq1}). Two particular examples are worth mentioning: $A^2_{-1}=H^2$ and $A^2_{-2}=\HD$, the Hardy and
Dirichlet spaces, respectively. See \cite{ZZ08} for more information about the ``generalized weighted Bergman
spaces'' $A^p_\alpha$.

Each space $A^2_\alpha$ is a Hilbert space with a certain choice of inner product. For example, if $\alpha>-1$, we
can choose $k=0$ in (\ref{eq1}) and simply use the natural inner product in $L^2(\D, dA_\alpha)$ for $A^2_\alpha$:
$$\langle f,g\rangle=\intD f(z)\overline{g(z)}\,dA_\alpha(z).$$

More generally, for any $\alpha>-2$, it is easy to show that an analytic function
$f(z)=\sum_{n=0}^\infty a_nz^n$ belongs to $A^2_\alpha$ if and only if
$$\sum_{n=0}^\infty\frac{|a_n|^2}{(n+1)^{\alpha+1}}<\infty.$$
Since
$$\frac{n!}{\Gamma(n+2+\alpha)}\sim\frac1{(n+1)^{\alpha+1}}$$
as $n\to\infty$, we see that
$$\langle f,g\rangle=\sum_{n=0}^\infty\frac{n!\,\Gamma(2+\alpha)}{\Gamma(n+2+\alpha)}\,a_n\overline b_n,\qquad
f(z)=\sum_{n=0}^\infty a_nz^n,\quad g(z)=\sum_{n=0}^\infty b_nz^n.$$
defines an inner product on $A^2_\alpha$. With this inner product, the functions
$$e_n(z)=\sqrt{\frac{\Gamma(n+2+\alpha)}{n!\,\Gamma(2+\alpha)}}\,z^n,\qquad n\ge0,$$
form an orthonormal basis for $A^2_\alpha$, which yields the reproducing kernel of $A^2_\alpha$ as follows:
\begin{equation}
K(z,w)=\sum_{n=0}^\infty e_n(z)\overline{e_n(w)}=\sum_{n=0}^\infty\frac{\Gamma(n+2+\alpha)}{n!\,\Gamma(2+\alpha)}
\,(z\overline w)^n=\frac1{(1-z\overline w)^{2+\alpha}}.
\label{eq2}
\end{equation}

Although all spaces $A^2_\alpha$, when $\alpha>-2$, have the same type of reproducing kernel as given
in (\ref{eq2}), their multiplier algebras depend on $\alpha$ in a critical way. It is well known that $\HM(A^2_\alpha)
=H^\infty$ for $\alpha\ge-1$. When $\alpha<-1$, $\HM(A^2_\alpha)$ is a proper sub-algebra
of $H^\infty$.

We will consider the defect operators
$$D_\varphi = D_\varphi^\alpha = \left(I-T_\varphi T_\varphi^*\right)^{1/2},\qquad
D_{\overline\varphi} = D_{\overline\varphi}^\alpha=\left(I-T_\varphi^*T_\varphi\right)^{1/2},$$
and the associated operators
$$E_\varphi=E_\varphi^\alpha=I-T_\varphi T_\varphi^*,\qquad
E_{\overline\varphi}=E_{\overline\varphi}^\alpha=I-T_\varphi^*T_\varphi,$$
where $\varphi\in\HM_1(A^2_\alpha)$ and $T_\varphi: A^2_\alpha\to A^2_\alpha$ is
the (contractive) multiplication operator.

Recall that
$$H^\alpha(\varphi)=H(T_\varphi),\qquad H^\alpha(\overline\varphi)=H(T^*_\varphi),$$
which are the generalized sub-Bergman Hilbert spaces defined in the Introduction. For any $\alpha>-2$, just like
the unweighted case $\alpha=0$, $H^\alpha(\varphi)$ is a reproducing kernel Hilbert space whose kernel
function is given by
\begin{equation}
K^{\alpha,\varphi}(z,w)=K^{\alpha,\varphi}_w(z)=\frac{1-\varphi(z)\overline{\varphi(w)}}{(1-z\overline w)^{2+\alpha}}.
\label{eq3}
\end{equation}
Similarly, $H^\alpha(\overline\varphi)$ is a reproducing kernel Hilbert space whose kernel function is given by
\begin{equation}
K^{\alpha,\overline\varphi}(z,w)=K^{\alpha,\overline\varphi}_w(z)=\intD\frac{1-|\varphi(u)|^2}
{(1-z\overline u)^{2+\alpha}(1-u\overline w)^{2+\alpha}}\,dA_\alpha(u).
\label{eq4}
\end{equation}
The spaces $H^\alpha(\varphi)$ and $H^\alpha(\overline\varphi)$ have been studied by
several authors, mostly in the case $\alpha\ge0$. See \cite{Su06, Chu19} for example.
We will generalize several results in the literature to weighted Bergman spaces $A^2_\alpha$ with
$\alpha>-1$.

\section{Complete Nevanlinna-Pick kernels}

In this section, we will determine exactly when the reproducing kernel function $K_w^\varphi$ in (\ref{eq3})
is a complete Nevanlinna-Pick (CNP) kernel. The following definition is from Theorem 8.2 in \cite{AM02}.

\begin{definition}\label{1}
Suppose $K=K(z,w)=K_w(z)$ is an irreducible kernel function on a set $\Omega\,$. $K$ is called a CNP kernel if there is
an auxiliary Hilbert space $\HL\,$, a function $b: \Omega \rightarrow \HL\,$, and a nowhere vanishing function $\delta$
on $\Omega$ such that
$$K_w(z) = \frac{\delta(z) \overline{\delta(w)}}{1-\langle b(z), b(w)\rangle},\qquad z,w\in\Omega.$$
\end{definition}

If $K$ is a CNP kernel, the corresponding Hilbert space $\HH(K)$ with kernel $K$ is called a CNP space.
CNP spaces share many properties with the Hardy space $H^2$, and they have been studied extensively in
the literature; see e.g. \cite{AM00, AHMR17, AHMR18, AHMR19, AHMR21} and the references therein for
recent developments. In 2020, Chu \cite{Chu20} determined which de Branges-Rovnyak spaces (sub-Hardy spaces)
have CNP kernel. We will characterize which sub-Bergman spaces have CNP kernel.

The reproducing kernel for the Hardy space $H^2$ is
$$K^{H^2}_w(z) = \frac{1}{1-z\overline{w}}.$$
If $\varphi \in H^\infty_1$ is not a constant, we let
$$H(K^{H^2}\circ \varphi) = \{f\circ \varphi: f \in H^2\}.$$
Then
$$K^{H^2}\circ \varphi(z,w) = K^{H^2}(\varphi(z),\varphi(w)) = \frac{1}{1-\varphi(z)\overline{\varphi(w)}}$$
is a kernel function and $C_\varphi: H^2 \rightarrow H(K^{H^2}\circ \varphi)$ defined by $C_\varphi f = f \circ \varphi$
is a unitary; see (\cite[P 71]{PR16}).

Given $a\in \D$ we let
$$\varphi_a(z) = \frac{a-z}{1-\overline az}$$
denote the M\"{o}bius map that interchanges the points $0$ and $a$. If we take $a=\varphi(0)$ and define
$$\psi(z)=\varphi_a(\varphi(z)),\qquad g(z)=\frac{\sqrt{1-|a|^2}}{1-\overline a\varphi(z)},$$
then an easy calculation shows that
\begin{equation}\label{eq5}
K^{\alpha,\psi}_w(z)= g(z)\,\overline{g(w)}\,K^{\alpha,\varphi}_w(z).
\end{equation}
See e.g. \cite[P 18]{LGR21}. So $K^{\alpha,\varphi}_w(z)$ is a CNP kernel if and only if $K^{\alpha,\psi}_w(z)$
is a CNP kernel.

The following result can be obtained from \cite[Theorem 6.28]{PR16}.

\begin{lemma}\label{2}
Let $\HH_1$ and $\HH_2$ be reproducing kernel Hilbert spaces of functions on a set $\Omega$ with reproducing
kernel $K_1$ and $K_2$, respectively. Let $\HF$ 
be a Hilbert space and $\Phi: \Omega\rightarrow\HB(\HF,\C)$ be a function. Then the following are equivalent:
\begin{enumerate}
\item $\Phi$ is a contractive multiplier from $\HH_1\otimes \HF$ to $\HH_2$,
\item $K_2(z,w)-K_1(z,w)\Phi(z)\Phi(w)^*$ is positive definite.
\end{enumerate}
\end{lemma}

We will use $\HM_1(\HH_1, \HH_2)$ to denote the set of contractive multipliers from $\HH_1$ to $\HH_2$.
When $\HH_1=\HH_2=\HH$, we will simplify the notation to $\HM_1(\HH)$.

\begin{lemma}\label{3}
Let $\varphi \in H^\infty_1$ be a non-constant function. Then
$$\HM_1(H(K^{H^2}\circ \varphi)) = \bigl\{f\circ \varphi: f \in \HM_1(H^2)\bigr\}.$$
\end{lemma}

\begin{proof}
This follows easily from the fact that $C_{\varphi}: H^2 \rightarrow H(K^{H^2}\circ \varphi)$ is a unitary.
\end{proof}

In what follows we will use the notation $K(z,w)\succeq 0$ or $0\preceq K(z,w)$ to mean that $K(z,w)$ is a
reproducing kernel, that is, $K(z,w)=\overline{K(w,z)}$ and it is positive-definite in the sense that
$$\sum_{i,j=1}^NP(z_i,z_j)c_i\overline c_j\ge0$$
for all $z_i\in\D$ and $c_i\in\C$, $1\le i\le N$, and $N\ge1$. We will begin with the following result for the
ordinary Bergman space, which illustrates the main techniques we use in this section.

\begin{theorem}\label{4}
Let $\varphi \in H^\infty_1$ and $\alpha=0$. Then $K^\varphi_w(z)=:K^{0,\varphi}_w(z)$ is a CNP kernel
if and only if $\varphi$ is a M\"{o}bius map.
\end{theorem}

\begin{proof}
If $\varphi$ is a M\"{o}bius map, say
$$\varphi = \zeta \frac{a-z}{1-\overline az},\qquad \zeta \in \T, a \in \D,$$
then it is easy to check that
$$K^\varphi_w(z) = \frac{1-|a|^2}{(1-\overline az)(1-a\overline{w})}\frac{1}{1-z\overline{w}},$$
which is clearly a CNP kernel.

Conversely, suppose $K^\varphi_w(z)$ is a CNP kernel. If $a = \varphi(0) \neq 0$, then we consider
$\psi(z)=\varphi_a(\varphi(z))$. By (\ref{eq5}), we have that $K^\psi_w(z)$ is a CNP kernel and
$\psi(0) = 0$. So we will assume that $\varphi$ also satisfies $\varphi(0) = 0$, which implies
$K^\varphi_0(z) = 1$ for all $z \in \D$.

It is well known that if a reproducing kernel function $K_w(z)=K(z,w)$ on $\D$ satisfies $K(z,0)=1$ for all $z\in\D$,
then it is a CNP kernel if and only if
$$1-\frac1{K(z,w)}\succeq 0.$$
See \cite{AM02} for example. Since
$$1 - \frac{1}{K^\varphi_w(z)} = 1 - \frac{(1-z\overline{w})^2}{1-\varphi(z)\overline{\varphi(w)}} =
\frac{2z\overline{w}-z^2\overline{w^2}-\varphi(z)\overline{\varphi(w)}}{1-\varphi(z)\overline{\varphi(w)}},$$
we have
$$\frac{1-\frac{z}{\sqrt{2}}\frac{\overline{w}}{\sqrt{2}}-\frac{\varphi(z)}{\sqrt{2}z}\frac{\overline{\varphi(w)}}{\sqrt{2}\overline{w}}}{1-\varphi(z)\overline{\varphi(w)}} \succeq 0.$$
It follows from this and Lemma \ref{2} that
\begin{equation}\label{eq6}
\Phi(z) = \left(\frac{z}{\sqrt{2}}, \frac{\varphi(z)}{\sqrt{2}z}\right) \in \HM_1\Bigl(H(K^{H^2}\circ \varphi) \otimes \C^2,
H(K^{H^2}\circ \varphi)\Bigr).
\end{equation}
Thus
$$\frac{z}{\sqrt{2}}\in \HM_1(H(K^{H^2}\circ \varphi)),\quad
\frac{\varphi(z)}{\sqrt{2}z} \in \HM_1(H(K^{H^2}\circ \varphi)).$$

Using $z/\sqrt{2} \in \HM_1(H(K^{H^2}\circ \varphi))$ and $1 \in H(K^{H^2}\circ \varphi)$, we can find
a function $h \in H^2$ such that
\begin{equation}\label{eq7}
\frac{z}{\sqrt{2}} = \frac{z}{\sqrt{2}} (1) = h(\varphi(z)),\quad z \in \D.
\end{equation}
Therefore $\varphi$ is injective, and by Lemma \ref{3}, $h \in \HM_1(H^2) = H^\infty_1$ and $h(0) = 0$.
Similarly, we deduce from $\varphi(z)/(\sqrt{2}\,z) \in \HM_1(H(K^{H^2}\circ \varphi))$ that $z/(2h) \in H^\infty_1$.
Then (\ref{eq6}) implies that
$$T: = \left(h, \frac{z}{2h}\right) \in \text{Mult}_1(H^2 \otimes \C^2, H^2).$$
Since
$$T^* \frac{1}{1-\overline{\lambda}z} = \left(\overline{h(\lambda)}, \overline{\frac{z}{2h}(\lambda)}\,\right)
\frac{1}{1-\overline{\lambda}z},$$
we conclude that
$$|h(\lambda)|^2 + \frac{|\lambda|^2}{4|h(\lambda)|^2} \leq 1, \qquad\lambda \in \D\setminus\{0\}.$$
Passing to boundary limits, we obtain
$$|h(\lambda)|^2 + \frac{1}{4|h(\lambda)|^2} \leq 1$$
for almost all $\lambda \in \T$. It follows that $|h(\lambda)| = \frac{1}{\sqrt{2}}$ for almost all $\lambda \in \T$.
Thus $\sqrt{2}h$ is an inner function. By Schwarz lemma, the inequality
$\sqrt2|h(z)|\le1$ together with $h(0)=0$ implies $\sqrt2|h(z)|\le|z|$ on $\D$.
This along with $z/(2h)\in H^\infty_1$ shows that
$$\frac1{\sqrt2}\le\left|\frac{\sqrt2\,h(z)}{z}\right|\le1,\qquad z\in\D,$$
which implies that the inner function $\sqrt2 h(z)/z$ has no zero inside $\D$ and has no singular factor. Therefore,
$\sqrt{2}h(z) = \zeta z$ for some $\zeta \in \T$. It then follows from (\ref{eq7}) that
$\varphi(z) = \overline{\zeta}z$, which finishes the proof of the theorem.
\end{proof}

The characterization of CNP kernels for the sub-$A^2_\alpha$ spaces $H^\alpha(\varphi)$ are more
subtle though. The results we obtain will depend on the range of the parameter $\alpha$.

\begin{theorem}\label{5}
Suppose $\varphi\in H^\infty_1$ and $-1<\alpha\le0$. Then the reproducing kernel of $H^\alpha(\varphi)$ in
(\ref{eq3}) is a CNP kernel if and only if $\varphi$ is a M\"{o}bius map.
\end{theorem}

\begin{proof}
The case $\alpha=0$ concerns the ordinary Bergman space, which is Theorem~\ref{4}.
So we assume $-1<\alpha<0$ for the rest of this proof.

First assume that $\varphi$ is a M\"{o}bius map, say $\varphi(z)=\zeta\,\frac{a-z}{1-\overline az}$ with $\zeta\in\T$ and
$a\in\D$. Then an easy computation shows that the reproducing kernel for $H^\alpha(\varphi)$ can be written as
$$K(z,w)=\frac{1-|a|^2}{(1-\overline az)(1-a\overline w)}\,\frac1{(1-z\overline w)^{1+\alpha}},$$
which is known to be a CNP kernel. See \cite{AM02}.

Next we assume that the kernel for $H^\alpha(\varphi)$ in (\ref{eq3})
is a CNP kernel. Once again, by considering $\psi(z)=\varphi_a\circ\varphi(z)$ with $a=\varphi(0)$ and
using (\ref{eq5}), we may assume that $\varphi(0)=0$.

When $\varphi(0)=0$, we have $K^{\alpha,\varphi}_0(z)=1$ for all $z\in\D$. In this case, it is known that the kernel
$K^{\alpha,\varphi}_w(z)$ is CNP if and only if $1-[1/K^{\alpha,\varphi}_w(z)]\succeq0$; see \cite{AM02} for example.
Since
\begin{align*}
1&-\frac1{K^{\alpha,\varphi}_w(z)}=1-\frac{(1-z\overline w)^{2+\alpha}}{1-\varphi(z)\overline{\varphi(w)}}\\
&=\left[sz\overline w-\sum_{n=2}^\infty\frac{s(s-1)\Gamma(n-s)}{n!\,\Gamma(2-s)}\,
z^n\overline w^n-\varphi(z)\overline{\varphi(w)}\,\right]\,\frac1{1-\varphi(z)\overline{\varphi(w)}},
\end{align*}
where $s=\alpha+2\in(1,2)$, we must have
$$\left[1-\sum_{n=2}^\infty\frac{(s-1)\,\Gamma(n-s)}{n!\,\Gamma(2-s)}\,z^{n-1}\overline w^{n-1}-
\frac{\varphi(z)}{\sqrt s z}\,\frac{\overline{\varphi(w)}}{\sqrt s\overline w}\right]\,\frac1{1-\varphi(z)\overline{\varphi(w)}}
\succeq 0.$$

Let
$$\Phi(z)=\left(\frac{\varphi(z)}{\sqrt sz},\sqrt{\frac{s-1}{2!}}\,z,\cdots,\sqrt{\frac{(s-1)\Gamma(n-s)}{n!\,\Gamma(2-s)}}\,
z^{n-1},\cdots\right).$$
By Lemma \ref{2}, we have
\begin{equation}
\Phi\in\HM_1\Bigl(H(K^{H^2}\circ\varphi)\otimes l^2, H(K^{H^2}\circ\varphi)\Bigr).
\label{eq8}
\end{equation}
Thus
$$\frac{\varphi(z)}{\sqrt sz},\quad\sqrt{\frac{(s-1)\Gamma(n-s)}{n!\,\Gamma(2-s)}}\,z^{n-1}\in
\HM_1\Bigl(H(K^{H^2}\circ\varphi)\Bigr),\qquad n\ge2.$$
It follows from
$$\sqrt{\frac{s-1}{2!}}\,z\in\HM_1\Bigl(H(K^{H^2}\circ\varphi\Bigr),\quad 1\in H(K^{H^2}\circ\varphi),$$
that there exists some function $h\in H^2$ such that
\begin{equation}
\sqrt{\frac{s-1}{2!}}\,z=\sqrt{\frac{s-1}{2!}}\,h(\varphi(z)),\qquad z\in\D.
\label{eq9}
\end{equation}
Therefore, $\varphi$ is injective, and by Lemma \ref{3},
$$\sqrt{\frac{s-1}{2!}}\,h\in\HM_1(H^2)=H^\infty_1$$
with $h(0)=0$. Then we also have
$$\sqrt{\frac{(s-1)\Gamma(n-s)}{n!\,\Gamma(2-s)}}\,z^{n-1}=
\sqrt{\frac{(s-1)\Gamma(n-s)}{n!\,\Gamma(2-s)}}\,h(\varphi(z))^{n-1},\qquad n\ge2.$$
Similarly, from $\varphi(z)/\sqrt sz\in\HM_1(H(K^{H^2}\circ\varphi))$ we obtain $z/\sqrt sh\in H^\infty_1$.

By (\ref{eq8}), we must have
\begin{align*}
T(z)&:=\left(\frac z{\sqrt sh},\sqrt{\frac{s-1}{2!}}\,h,\cdots,\sqrt{\frac{s-1)\Gamma(n-s)}{n!\,\Gamma(2-s)}}\,h^{n-1},
\cdots\right)\\
&\in\HM_1(H^2\otimes l^2, H^2).
\end{align*}
Note that
\begin{align*}
&T^*\frac1{1-\overline\lambda z}=\\
&\quad\left(\overline{\frac z{\sqrt sh}(\lambda)},\sqrt{\frac{s-1}{2!}}\,\overline{h(\lambda)},\cdots,
\sqrt{\frac{(s-1)\Gamma(n-s)}{n!\,\Gamma(2-s)}}\,\overline{h^{n-1}(\lambda)}\,\right)\,\frac1{1-\overline\lambda z}.
\end{align*}
It follows that
$$\frac{|\lambda|^2}{s|h(\lambda)|^2}+\sum_{n=2}^\infty\frac{(s-1)\Gamma(n-2)}{n!\,\Gamma(2-s)}
\,|h(\lambda)|^{2n-2}\le1,\qquad\lambda\in\D\setminus\{0\}.$$
Passing to radial limits, we obtain
$$\frac1{s|h(\lambda)|^2}+\sum_{n=2}^\infty\frac{(s-1)\Gamma(n-s)}{n!\,\Gamma(2-s)}\,|h(\lambda)|^2\le1$$
or
$$1+\sum_{n=2}^\infty\frac{s(s-1)\Gamma(n-s)}{n!\,\Gamma(2-s)}\,|h(\lambda)|^{2n}\le s|h(\lambda)|^2$$
for almost all $\lambda\in\T$. We necessarily have $|h(\lambda)|^2\le1$. Comparing the above inequality with the
classical Taylor series
$$(1-x)^s=1-sx+\sum_{n=2}^\infty\frac{s(s-1)\Gamma(n-s)}{n!\,\Gamma(2-s)}\,x^n,\qquad x\in(-1,1),$$
we obtain $(1-|h(\lambda)|^2)^s\le0$ for almost all $\lambda\in\T$, so $h$ is an inner function. This together with
$z/\sqrt sh\in H^\infty_1$ implies that $h(z)=\zeta z$ for some constant $\zeta\in\T$. By (\ref{eq9}), we have
$\varphi(z)=\overline\zeta\,z$. This completes the proof of the theorem.
\end{proof}

Note that, in the case when $\alpha=-1$, a characterization for $\varphi\in H^\infty_1$ was obtained in \cite{Chu20}
in order for the kernel
$$K(z,w)=\frac{1-\varphi(z)\overline{\varphi(w)}}{(1-z\overline w)^{2+\alpha}}=
\frac{1-\varphi(z)\overline{\varphi(w)}}{1-z\overline w}$$
to be CNP. The necessary and sufficient condition for $\varphi$ is the following: there exists a function
$h\in H^\infty_1$ such that $\psi(z)=zh(\psi(z))$, where $\psi(z)=\varphi_a(\varphi(z))$ with $a=\varphi(0)$.

When $-2<\alpha<-1$, we have the following result.

\begin{theorem}\label{6}
Suppose $-2<\alpha<-1$ and $\varphi\in\HM_1(A^2_\alpha)$. Let $a=\varphi(0)$ and $\psi=\varphi_a\circ\varphi$.
Then the function
$$K^{\alpha,\varphi}_w(z)=\frac{1-\varphi(z)\overline{\varphi(w)}}{(1-z\overline w)^{2+\alpha}}$$
is a CNP kernel if and only if there exists
$$h=(h_1, h_2,\cdots,h_n,\cdots)\in\HM_1(H^2, H^2\otimes l^2)$$
such that
$$\psi(z)=\sum_{n=1}^\infty\sqrt{\frac{(2+\alpha)\Gamma(n-\alpha-2)}{n!\,\Gamma(-1-\alpha)}}\,z^nh_n(\psi(z))$$
on $\D$.
\end{theorem}

\begin{proof}
Recall from (\ref{eq5}) that $K^{\alpha,\varphi}_w(z)$ is a CNP kernel if and only if $K^{\alpha,\psi}_w(z)$ is
a CNP kernel. So we will assume that $\varphi(0)=0$. In this case, we have $K^{\alpha,\varphi}_0(z)=1$ for all
$z\in\D$ and $1-[1/K^{\alpha,\varphi}_w(z)]\succeq 0$.

Let $s=\alpha+2$ and write
\begin{align*}
1-\frac1{K^{\alpha,\varphi}_w(z)}&=1-\frac{(1-z\overline w)^s}{1-\varphi(z)\overline{\varphi(w)}}\\
&=\left(\sum_{n=1}^\infty\frac{s\Gamma(n-s)}{n!\,\Gamma(1-s)}\,z^n\overline w^n-\varphi(z)\overline{\varphi(w)}\right)
\,\frac1{1-\varphi(z)\overline{\varphi(w)}}.
\end{align*}
Since $1/(1-\varphi(z)\overline{\varphi(w)})$ is a CNP kernel, it follows
from Theorem 8.57 of \cite{AM02} that $1-[1/K^{\alpha,\varphi}_w(z)]\succeq0$ if and only if there exists
$$\Phi=(\varphi_n)\in\HM_1\Bigl(H(K^{H^2}\circ\varphi), H(K^{H^2}\circ\varphi)\otimes l^2\Bigr)$$
such that
$$\varphi(z)=\sum_{n=1}^\infty\sqrt{\frac{s\Gamma(n-s)}{n!\,\Gamma(1-s)}}\,z^n\varphi_n(z).$$
By Lemma~\ref{3}, there exist $h=(h_n)\subset H^\infty_1$ such that $\varphi_n(z)=h_n(\varphi(z))$
for all $n$ and $h\in\text{Mult}_1(H^2, H^2\otimes l^2)$. This proves the desired result.
\end{proof}

For an example of a CNP kernel $K^{\alpha,\varphi}_w(z)$ when $-2<\alpha<-1$, fix any positive integer $n$ and
consider
$$\varphi(z)=\sqrt{\frac{(2+\alpha)\Gamma(n-2-\alpha)}{n!\,\Gamma(-1-\alpha)}}\,z^n.$$
It is easy to see that $\varphi\in\HM_1(A^2_\alpha)$ and, by the theorem above, $K^{\alpha,\varphi}_w(z)$
is a CNP kernel.

The case $\alpha>0$ remains unsettled. In this case, the identity function $\varphi(z)=z$ belongs to $H^\infty_1=
\HM_1(A^2_\alpha)$, but
$$K^{\alpha,\varphi}_w(z)=\frac{1-z\overline w}{(1-z\overline w)^{2+\alpha}}=\frac1{(1-z\overline w)^{1+\alpha}}$$
is NOT a CNP kernel; see \cite{AM02}. We do not know if there exists {\it any} $\varphi\in H^\infty_1$ such that
$K^{\alpha,\varphi}_w(z)$ is a CNP kernel.

\section{Compactness of defect operators}

In this section we will characterize functions $\varphi\in H^\infty_1$ such that the defect operators $D^\alpha_\varphi$
and $D^\alpha_{\overline\varphi}$, where $\alpha>-1$, are compact. The following result follows from I-9 of \cite{Sa94}.

\begin{lemma}\label{8}
Let $\alpha > -1$, $\varphi\in H^\infty_1$, and $M^\alpha(\varphi) = \varphi A^2_\alpha$. Then
$$H^\alpha(\varphi) \cap M^\alpha(\varphi) = \varphi H^\alpha(\overline{\varphi}).$$
\end{lemma}

The following result was proved in \cite{Su06} for $\alpha\ge0$.
We now show that the assumption $\alpha \geq 0$ can be relaxed to $\alpha > -1$.

\begin{lemma}\label{9}
Let $\alpha > -1$ and $\varphi\in H^\infty_1$. If $\varphi$ is a finite Blaschke product, then
$$H^\alpha(\overline\varphi) = H^\alpha(\varphi) = A^2_{\alpha-1}.$$
\end{lemma}

\begin{proof}
Since the result is already known for $\alpha\ge0$, we will assume that $-1 < \alpha < 0$.

By the definition of $A^2_{\alpha-1}$, it is not hard to see that any function that is analytic on the closed unit disk is a multiplier of $A^2_{\alpha-1}$. In particular, $T_\varphi$ is a bounded operator on $A^2_{\alpha-1}$. If $\|T_{\varphi}\|_{B(A^2_{\alpha-1})} = C < \infty$, then
$$(I- T_{\varphi} T_{\varphi}^*/C^2)K^{\alpha-1}_w(z) = \frac{1-\varphi(z)\overline{\varphi(w)}/C^2}{(1-z\overline{w})^{1+\alpha}} \succeq 0.$$
Thus by the Schur product theorem (\cite{PR16}),
$$(1-\varphi(z)\overline{\varphi(w)}/C^2)\frac{(1-\varphi(z)\overline{\varphi(w)})}{(1-z\overline{w})^{2+\alpha}} = \frac{1-\varphi(z)\overline{\varphi(w)}/C^2}{(1-z\overline{w})^{1+\alpha}}\frac{1-\varphi(z)\overline{\varphi(w)}}{1-z\overline{w}} \succeq 0.$$
It follows that $\varphi/C$ is a contractive multiplier of $H^{\alpha}(\varphi)$. Thus $\varphi H^{\alpha}(\varphi) \subseteq H^{\alpha}(\varphi)$. Combining this with $H^{\alpha}(\varphi) \subseteq A^2_\alpha$, we obtain
$$\varphi H^{\alpha}(\varphi) \subseteq H^{\alpha}(\varphi) \cap \varphi A^2_\alpha = H^{\alpha}(\varphi) \cap M^\alpha(\varphi).$$
By Lemma~\ref{8}, we also have $\varphi H^{\alpha}(\varphi) \subseteq \varphi H^{\alpha}(\overline{\varphi})$,
so $H^{\alpha}(\varphi) \subseteq H^{\alpha}(\overline{\varphi})$.

To finish the proof, we note $H^{\alpha}(\varphi) = A^2_{\alpha-1}$ (\cite{Su06}) and use the fact that the
subnormality of $T_\varphi$ gives $H^{\alpha}(\overline{\varphi}) \subseteq H^{\alpha}(\varphi)$ in general.
\end{proof}

\begin{lemma}
Let $\varphi$ be a non-constant function in $H^\infty_1$. Then the following conditions are equivalent.
\begin{itemize}
\item[(a)] $\varphi$ is a finite Blaschke product.
\item[(b)] $1-|\varphi(z)|^2\to0$ as $|z|\to1^-$.
\item[(c)] $(1-|\varphi(z)|^2)/(1-|z|^2)$ is bounded both above and below on $\D$.
\end{itemize}
\label{11}
\end{lemma}

\begin{proof}
The equivalence of (a) and (c) was proved in \cite{Zhu03}. It is trivial that (c) implies (b).

If (b) holds, then $|\varphi(z)|\to1$ uniformly as $|z|\to1^-$, so $\varphi$ is an inner function. It is clear that
$\varphi$ cannot have infinitely many zeros. If $\varphi$ contains a singular inner factor $S$, then there exists at
least one point $\zeta\in\T$ such that $S(z)\to0$ as $z$ approaches $\zeta$ radially, which contradicts with the
limit $|\varphi(z)|\to1$ as $|z|\to1^-$. Thus $\varphi$ cannot contain any singular inner factor. Hence $\varphi$
must be a finite Blaschke product. This shows that (b) implies (a) and completes the proof of the lemma.
\end{proof}

\begin{lemma}
Suppose $\alpha>-1$ and $T: A^2_\alpha\to A^2_\alpha$ is a bounded linear operator. If the range of $T$ is
contained in $A^2_\gamma$ for some $\gamma<\alpha$, then $T$ belongs to the Schatten class $S_p$ for
all $p>2/(\alpha-\gamma)$.
\label{12}
\end{lemma}

\begin{proof}
It is well known that if $\gamma<\alpha$, then $A^2_\gamma\subset A^2_\alpha$, and the inclusion mapping
$i: A^2_\gamma\to A^2_\alpha$ is bounded. If $T$ maps $A^2_\alpha$ into $A^2_\gamma$, then by the closed
graph theorem, there exists a constant $C>0$ such that $\|Tf\|_{A^2_\gamma}\le C\|f\|_{A^2_\alpha}$ for all
$f\in A^2_\alpha$, that is, $T$ can be thought of as a bounded linear operator from $A^2_\alpha$ into
$A^2_\gamma$. We can then write $T=iT$ and $T^*T=T^*(i^*i)T$.

The operator $i^*i: A^2_\gamma\to A^2_\gamma$ is positive. With respect to the monomial orthonormal basis
$\{e_n=c_nz^n\}$ of $A^2_\gamma$ from Section 2, the operator $i^*i$ is diagonal with the corresponding
eigenvalues given by
$$\langle i^*ie_n, e_n\rangle_{A^2_\gamma}=c_n^2\langle z^n, z^n\rangle_{A^2_\alpha}
=\frac{\Gamma(n+2+\gamma)}{n!\,\Gamma(2+\gamma)}\,\frac{n!\,\Gamma(2+\alpha)}{\Gamma(n+2+\alpha)}
\sim\frac1{(n+1)^{\alpha-\gamma}}$$
as $n\to\infty$. This shows that $i^*i$ belongs to the Schatten class $S_p$ of $A^2_\gamma$ for all $p$ with
$p(\alpha-\gamma)>1$. Thus $T$ belongs to the Schatten class $S_p$ of $A^2_\alpha$ whenever
$p>2/(\alpha-\gamma)$.
\end{proof}

Note that the result above remains true even if the parameters $\alpha$ and $\gamma$ fall below $-1$,
although the proof needs to be modified. Details are omitted. We now prove the main results of this section
in the next two theorems.

Recall that
$$D_\varphi^\alpha = \left(I-T_\varphi T_\varphi^*\right)^{1/2},\qquad
D_{\overline\varphi}^\alpha=\left(I-T_\varphi^*T_\varphi\right)^{1/2}$$
are the defect operators, and
$$E_\varphi^\alpha=I-T_\varphi T_\varphi^*,\qquad
E_{\overline\varphi}^\alpha=I-T_\varphi^*T_\varphi.$$

\begin{theorem}
Suppose $\alpha>-1$ and $\varphi\in H^\infty_1$. Then the following conditions are equivalent.
\begin{itemize}
\item[(a)] The defect operator $D^\alpha_\varphi$ is compact on $A^2_\alpha$.
\item[(b)] The function $\varphi$ is a finite Blaschke product.
\item[(c)] The space $H^\alpha(\varphi)$ equals $A^2_{\alpha-1}$.
\item[(d)] The space $H^\alpha(\varphi)$ is contained in $A^2_{\alpha-1}$.
\end{itemize}
\label{13}
\end{theorem}

\begin{proof}
To prove (a) implies (b), we consider the normalized reproducing kernels
$$k_a(z)=\frac{K_a(z)}{\|K_a\|}=\frac{K(z,a)}{\sqrt{K(a,a)}}=\frac{(1-|a|^2)^{(2+\alpha)/2}}{(1-z\overline a)^{2+\alpha}}$$
for $A^2_\alpha$. It is easy to see that $k_a\to0$ weakly in $A^2_\alpha$ as $|a|\to1^-$. If $D^\alpha_\varphi$ is
compact, then so is $E^\alpha_\varphi$, which implies that $\langle E^\alpha_\varphi k_a, k_a\rangle\to0$ as
$|a|\to1^-$. It is easy to see that $T^*_\varphi k_a=\overline{\varphi(a)}\,k_a$, so we have
\begin{equation}
\langle E^\alpha_\varphi k_a, k_a\rangle=\langle(I-T_\varphi T^*_\varphi)k_a, k_a\rangle
=1-\langle T^*_\varphi k_a, T^*_\varphi k_z\rangle=1-|\varphi(a)|^2.
\label{eq11}
\end{equation}
Thus the compactness of $D^\alpha_\varphi$ implies $1-|\varphi(a)|^2\to0$ as $|a|\to1^-$, which, according to
Lemma~\ref{11}, shows that $\varphi$ is a finite Blaschke product. This proves (a) implies (b).

Lemma~\ref{9} states that (b) implies (c). It is trivial that (c) implies (d). It follows from Lemma~\ref{12} that
(d) implies (a). This completes the proof of the theorem.
\end{proof}

\begin{theorem}
Suppose $\alpha>-1$ and $\varphi\in H^\infty_1$. Then the following conditions are equivalent.
\begin{itemize}
\item[(a)] The defect operator $D^\alpha_{\overline\varphi}$ is compact on $A^2_\alpha$.
\item[(b)] The function $\varphi$ is a finite Blaschke product.
\item[(c)] The space $H^\alpha(\overline\varphi)$ equals $A^2_{\alpha-1}$.
\item[(d)] The space $H^\alpha(\overline\varphi)$ is contained in $A^2_{\alpha-1}$.
\end{itemize}
\label{14}
\end{theorem}

\begin{proof}
First assume that condition (a) holds. Taking the square of $D^\alpha_{\overline\varphi}$, we see that the
Toeplitz operator $T_{1-|\varphi|^2}$ (with nonnegative
symbol) is compact on $A^2_\alpha$. It follows from Corollary 7.9 of \cite{Zhu07} that for any positive $r>0$ we have
$$\lim_{|a|\to1^-}\frac1{A_\alpha(D(a,r))}\int_{D(a,r)}(1-|\varphi(z)|^2)\,dA_\alpha(z)=0,$$
where $D(a,r)=\{z\in\D: \beta(z,a)<r\}$ is the Bergman metric ball with center $a$ and radius $r$, and $A_\alpha(D(a,r))$ is the $dA_\alpha$ measure of $D(a,r)$.
Equivalently,
\begin{equation}
\lim_{|a|\to1^-}\frac1{A_\alpha(D(a,r))}\int_{D(a,r)}|\varphi(z)|^2\,dA_\alpha(z)=1.
\label{eq7}
\end{equation}
We claim that this implies $|\varphi(z)|^2\to1$ uniformly as $|z|\to1^-$. In fact, if this conclusion is not true, then
there exist a constant $\sigma\in(0,1)$ and a sequence $\{a_n\}$ in $\D$ such that $|a_n|\to1$ as $n\to\infty$ and
$|\varphi(a_n)|<\sigma$ for all $n\ge1$.

If $z\in D(a_n, r)$, then by Theorem 5.5 of \cite{Zhu07},
$$|\varphi(z)|\le|\varphi(z)-\varphi(a_n)|+|\varphi(a_n)|
\le\|\varphi\|_{\HB}\,\beta(z,a_n)+\sigma<\|\varphi\|_{\HB}r+\sigma,$$
where
$$\|\varphi\|_{\HB}=\sup_{z\in\D}(1-|z|^2)|\varphi'(z)|$$
is Bloch semi-norm of $\varphi$ (recall that every function in $H^\infty$ belongs to the Bloch space).
If we use a sufficiently small radius $r$ such that the constant $\delta=\|\varphi\|_{\HB}r+\sigma<1$, then
$$\frac1{A_\alpha(D(a_n,r))}\int_{D(a_n,r)}|\varphi(z)|^2\,dA_\alpha(z)\le\delta^2<1$$
for all $n\ge1$. This is a contradiction to (\ref{eq7}).

Thus we must have $|\varphi(z)|^2\to1$ uniformly as $|z|\to1^-$. By Lemma~\ref{11}, $\varphi$ is a finite
Blaschke product. This proves that (a) implies (b).

It follows from Lemma \ref{9} that (b) implies (c). It is trivial that (c) implies (d). That (d) implies (a) follows
from Lemma~\ref{12}.
\end{proof}

It follows from the proof of the theorem above that, for $\alpha>-1$, $k>0$, and $\varphi\in H^\infty_1$,
the Toeplitz operator $T_{(1-|\varphi|^2)^k}$ is compact on $A^2_\alpha$ if and only if
$\varphi$ is a finite Blaschke product.

\section{The range of $I-T_\varphi T_{\varphi}^*$ and $I-T_{\varphi}^*T_\varphi$}

In this section we study the range of the operators $E^\alpha_\varphi$ and $E^\alpha_{\overline\varphi}$.
The special case $\alpha=0$ was considered in \cite{Zhu03}. It is clear that $D^\alpha_\varphi$ is compact on
$A^2_\alpha$ if and only if $E^\alpha_\varphi$ is compact on $A^2_\alpha$. Similarly, $D^\alpha_{\overline\varphi}$
is compact on $A^2_\alpha$ if and only if $E^\alpha_{\overline\varphi}$ is compact on $A^2_\alpha$.

\begin{prop}\label{15}
Suppose $\alpha>-1$ and $\varphi$ is a finite Blaschke product. Then
\begin{align}
A^2_{\alpha-1}&=\left\{f(z)=\intD\frac{1-|\varphi(w)|^2}{(1-z\overline w)^{2+\alpha}}\,g(w)\,dA_\alpha(w):
g\in A^2_{\alpha+1}\right\}\label{eq13}\\
&=\left\{f(z)=\intD\frac{1-|\varphi(w)|^2}{(1-z\overline w)^{2+\alpha}}\,g(w)\,dA_\alpha(w):
g\in L^2(\D, dA_{\alpha+1})\right\}.\label{eq14}
\end{align}
\end{prop}

\begin{proof}
Let
$$dA_{\varphi,\alpha}(z)=(1-|\varphi(z)|^2)\,dA_\alpha(z)$$
and let $A^2_{\varphi,\alpha}$ denote the space
of analytic functions in $L^2(\D, dA_{\varphi,\alpha})$. It follows from Lemma~\ref{11} that
$$L^2(\D, dA_{\varphi,\alpha})=L^2(\D, dA_{\alpha+1}),\qquad A^2_{\varphi,\alpha}=A^2_{\alpha+1},$$
with equivalent norms. Consider the operator $S_\varphi: A^2_{\varphi,\alpha}\to A^2_\alpha$ defined by
\begin{equation}
S_\varphi f(z)=P_\alpha[(1-|\varphi|^2)f](z)=\intD\frac{1-|\varphi(w)|^2}{(1-z\overline w)^{2+\alpha}}
\,f(w)\,dA_\alpha(w),
\label{eq15}
\end{equation}
where $P_\alpha: L^2(\D, dA_\alpha)\to A^2_\alpha$ is the orthogonal projection. It is clear that $S_\varphi$ is
simply the operator $E^\alpha_{\overline\varphi}$ with its domain extended to the larger space $A^2_{\varphi,\alpha}$.

Now the first desired equality (\ref{eq13}) follows from the proof of Proposition 3.5
in \cite{Zhu96}, word by word, together with the fact that $H^\alpha(\overline\varphi)=A^2_{\alpha-1}$ from
the previous section. The second equality (\ref{eq14}) follows from the same argument by replacing the
operator $S_\varphi$ above by its extension $S_\varphi: L^2(\D, dA_{\varphi,\alpha})\to A^2_\alpha$, still defined
by (\ref{eq15}). We leave the routine details to the interested reader.
\end{proof}

\begin{lemma}
If $\alpha>-1$ and $\varphi$ is a finite Blaschke product, then
$$E^\alpha_\varphi(A^2_\alpha)=E^\alpha_{\overline\varphi}(A^2_\alpha)=A^2_{\alpha-2}.$$
\label{16}
\end{lemma}

\begin{proof}
As a Toeplitz operator on $A^2_\alpha$, we can write
$$E^\alpha_{\overline\varphi}f(z)=\intD\frac{1-|\varphi(w)|^2}{(1-z\overline w)^{2+\alpha}}\,f(w)\,dA_\alpha(w),
\qquad f\in A^2_\alpha.$$
It follows that
$$(E^\alpha_{\overline\varphi}f)'(z)=\intD\frac{\Phi(w)}{(1-z\overline w)^{3+\alpha}}\,f(w)\,dA_{\alpha+1}(w)
=P_{\alpha+1}(\Phi f)(z),$$
where $P_{\alpha+1}: L^2(\D, dA_{\alpha+1})\to A^2_{\alpha+1}$ is the orthogonal projection and
$$\Phi(w)=\frac{(\alpha+1)\overline w(1-|\varphi(w)|^2)}{1-|w|^2}.$$
By Lemma \ref{11}, $\Phi\in L^\infty(\D)$. It follows from Theorem 3.11 of \cite{Zhu07} that $P_{\alpha+1}$
maps $L^2(\D, dA_\alpha)$ boundedly to $A^2_\alpha$. Therefore, $f\in A^2_\alpha$ implies
$(E^\alpha_{\overline\varphi}f)'\in A^2_{\alpha}$, which is clearly equivalent to $E^\alpha_{\overline\varphi}f
\in A^2_{\alpha-2}$. This proves that $E^\alpha_{\overline\varphi}$ maps $A^2_\alpha$ into $A^2_{\alpha-2}$.

To show that the mapping $E^\alpha_{\overline\varphi}: A^2_\alpha\to A^2_{\alpha-2}$ is onto, we switch from the
ordinary derivative $(E^\alpha_{\overline\varphi}f)'$ to a certain fractional radial differential operator $R$
($=R^{2+\alpha,1}$ using the notation from \cite{ZZ08}) of order $1$:
$$RE^\alpha_{\overline\varphi}f(z)=\intD\frac{1-|\varphi(w)|^2}{(1-z\overline w)^{3+\alpha}}\,f(w)\,dA_\alpha(w)$$
It is still true that $E^\alpha_{\overline\varphi}f\in A^2_{\alpha-2}$ if and only if
$RE^\alpha_{\overline\varphi}f\in A^2_\alpha$. See \cite{ZZ08}.

Fix any function $g\in A^2_{\alpha-2}$. Then the function $Rg$ belongs to $A^2_\alpha$. It follows from
Proposition~\ref{15}, with $\alpha$ in (\ref{eq13}) and (\ref{eq14}) replaced by $\alpha+1$, that there exists
a function $h\in L^2(\D, dA_{\alpha+2})$ such that
$$Rg(z)=\intD\frac{1-|\varphi(w)|^2}{(1-z\overline w)^{3+\alpha}}\,h(w)\,dA_{\alpha+1}(w).$$
Applying the inverse of $R$ to both sides, we obtain
$$g(z)=\intD\frac{1-|\varphi(w)|^2}{(1-z\overline w)^{2+\alpha}}\,h(w)\,dA_{\alpha+1}(w).$$
Let $\widetilde h(w)=(1-|w|^2)h(w)$. Then $\widetilde h\in L^2(\D, dA_\alpha)$ and
$$g(z)=\intD\frac{1-|\varphi(w)|^2}{(1-z\overline w)^{2+\alpha}}\,\widetilde h(w)\,dA_\alpha(w).$$
By Proposition~\ref{15} again, there exists a function $f\in A^2_\alpha$ such that
$$g=\intD\frac{1-|\varphi(w)|^2}{(1-z\overline w)^{2+\alpha}}\,f(w)\,dA_\alpha(w),$$
or $g=E^\alpha_{\overline\varphi}f$. Thus we have shown that $E^\alpha_{\overline\varphi}(A^2_\alpha)
=A^2_{\alpha-2}$.

Next we show $E^\alpha_\varphi(A^2_\alpha)=A^2_{\alpha-2}$. Note that
$T_\varphi$ is a Fredholm operator. So $\varphi A^2_\alpha$ is closed in $A^2_\alpha$,
$\ker (T_\varphi^*) = A^2_\alpha \ominus\varphi A^2_\alpha$, and
$A^2_\alpha = (A^2_\alpha \ominus\varphi A^2_\alpha)\oplus \varphi A^2_\alpha$. Since
$$(I-T_\varphi T_\varphi^*)\varphi f =\varphi (I-T_\varphi^*T_\varphi) f, \quad f \in A^2_\alpha,$$
it follows that
$$E^\alpha_\varphi(A^2_\alpha)=(A^2_\alpha\ominus\varphi A^2_\alpha)\oplus \varphi E^\alpha_{\overline\varphi}(A^2_\alpha) = (A^2_\alpha\ominus\varphi A^2_\alpha)\oplus \varphi A^2_{\alpha-2}.$$
Since $A^2_\alpha \ominus\varphi A^2_\alpha$ consists of the reproducing kernels or the derivative of the reproducing kernels in $A^2_\alpha$, we have $A^2_\alpha \ominus\varphi A^2_\alpha \subseteq A^2_{\alpha-2}$. Also
$$\dim (A^2_\alpha \ominus\varphi A^2_\alpha) = \dim (A^2_{\alpha-2} \ominus\varphi A^2_{\alpha-2}).$$
Thus we obtain that
$$E^\alpha_\varphi(A^2_\alpha) = (A^2_{\alpha-2} \ominus\varphi A^2_{\alpha-2}) + \varphi A^2_{\alpha-2} = A^2_{\alpha-2},$$
completing the proof of the lemma.
\end{proof}

We can now prove the main result of this section, namely, the next two theorems.

\begin{theorem}
Suppose $\alpha>-1$ and $\varphi\in H^\infty_1$. Then the following conditions are equivalent.
\begin{itemize}
\item[(a)] The operator $E^\alpha_{\varphi}$ is compact on $A^2_\alpha$.
\item[(b)] The function $\varphi$ is a finite Blaschke product.
\item[(c)] The range of $E^\alpha_\varphi$ equals $A^2_{\alpha-2}$.
\item[(d)] The range of $E^\alpha_\varphi$ is contained in $A^2_{\alpha-2}$.
\end{itemize}
\label{17}
\end{theorem}

\begin{proof}
Since $E^\alpha_\varphi=(D^\alpha_\varphi)^2$, the operator $E^\alpha_\varphi$ is compact if and only if
$D^\alpha_\varphi$ is compact. Thus the equivalence of (a) and (b) follows from Theorem~\ref{13}.

Lemma~\ref{16} shows that (b) implies (c). It is trivial that (c) implies (d). Finally, that (d) implies (a) follows
from Lemma~\ref{12}.
\end{proof}

\begin{theorem}
Suppose $\alpha>-1$ and $\varphi\in H^\infty_1$. Then the following conditions are equivalent.
\begin{itemize}
\item[(a)] The operator $E^\alpha_{\overline\varphi}$ is compact on $A^2_\alpha$.
\item[(b)] The function $\varphi$ is a finite Blaschke product.
\item[(c)] The range of $E^\alpha_{\overline\varphi}$ equals $A^2_{\alpha-2}$.
\item[(d)] The range of $E^\alpha_{\overline\varphi}$ is contained in $A^2_{\alpha-2}$.
\end{itemize}
\label{18}
\end{theorem}

\begin{proof}
It is similar to the proof of Theorem~\ref{17}.
\end{proof}

Finally, we note that the main results of this and the previous section cannot be extended to the Hardy space $H^2$
(the case $\alpha=-1$). For example, in this case, if $\varphi(z)=z$, then $I-T_{\overline\varphi}T_{\varphi}=0$ and
$I-T_\varphi T_{\overline\varphi}$ is a rank-one operator. More generally, if $\varphi$ is any inner function, then
$I-T_{\overline\varphi}T_\varphi=0$.

\end{document}